\documentclass{amsart}
\usepackage{amsthm,amsmath,amsfonts,amssymb,array,enumerate,verbatim}
\usepackage[all]{xy}
\usepackage{url}
\usepackage[small,nohug,heads=vee]{diagrams}
\diagramstyle[labelstyle=\scriptstyle]

\newcommand{\CC}{\mathbb{C}}

\newcommand{\QQ}{\mathbb{Q}}

\newcommand{\cW}{\mathcal{W}}

\DeclareMathOperator{\Hom}{Hom}

\DeclareMathOperator{\Ind}{Ind}
\DeclareMathOperator{\cInd}{c-Ind}

\DeclareMathOperator{\Rep}{Rep}

\DeclareMathOperator{\Irrep}{Irrep}

\DeclareMathOperator{\Res}{Res}

\newtheorem{thm}{Theorem}[section]
\newtheorem{lemma}[thm]{Lemma}
\newtheorem{prop}[thm]{Proposition}
\newtheorem{cor}[thm]{Corollary}

\newtheorem{rmk}[thm]{Remark}
\newtheorem{obs}[thm]{Observation}

\title{Functorial Zeta Integrals}
\author{Gil Moss}
\date{\today}
\begin{document}
\maketitle

In this expository note, we translate the proof of \cite[Prop 2.11]{jps2} into the purely functorial language of parabolic induction-restriction in \cite{b-z,b-zI}. This proposition gives the functional equation of the nonarchimedean local Rankin-Selberg Euler factors introduced by Jacquet, Piatetski-Shapiro, and Shalika in \cite{jps2}. This new language gives a clearer presentation of the ideas.  After writing this note the author found the book \cite{explicit}, which employs similar techniques although they do not use the language of \cite{b-z,b-zI} as concretely as we do here, and they do not seem interested in extending to more arbitrary coefficient rings. The proof presented here works over arbitrary fields of characteristic $\ell\neq p$.

Let $F$ be a finite extension of $\QQ_p$ and let $k$ be a field of characteristic $\ell\neq p$. For any $n$ define $G_n$ and $P_n$ to be the groups $GL_n(F)$ and $\left(\begin{smallmatrix} g & x \\ 0 & 1 \end{smallmatrix}\right)$. Choose a nontrivial additive character $\psi:F\rightarrow k^{\times}$, extend it to unipotent upper triangular matrices as usual.

Let $\tau$ be in $\Rep_k(G_r)$ be irreducible generic and let $\tau'$ in $\Rep_k(G_t)$ be irreducible generic with $t<r$. Denote their Whittaker models by $\cW(\tau,\psi)$, and $\cW(\tau',\psi)$. We will identify $\tau$ with $\cW(\tau,\psi)$. We also recall that for a character $\chi\circ \det$, $\cW(\tau'\chi,\psi) = \{W'(g)\chi(\det g):W\in \cW(\tau',\psi)\}\cong \cW(\tau',\psi)\otimes(\chi\circ\det)$.

We make heavy use of the functors $\Phi^+$, $\Phi^-$, $\Psi^+$, $\Psi^-$ and $\hat{\Phi}^+$ constructed in \cite{b-z,b-zI}. See also \cite{cogshap} for a nice exposition of their definition and properties, or \cite{moss1} for a shorter summary.

The author would like to thank his thesis advisor David Helm for his helpful conversations on this subject, and Robert Kurinczuk and Nadir Matringe for their interest in an earlier version of these notes posted on the author's webpage in July 2014.

\section{If $Z_s(W,W')$ is rational, it naturally lives in a nice Hom space}

Treating the letter $s$ as an indeterminate, we can define a character $\chi_s:G\rightarrow k[s,s^{-1}]^{\times}$ by $g\mapsto s^{v_F(\det g)}$. Then we can define a bilinear pairing $Z_s:\cW(\tau) \times \cW(\tau'\otimes \chi_s)\rightarrow k((s))$ as follows:

$$Z_s(W,W') = \int_{N_t\backslash G_t} W\left(\begin{smallmatrix}g&0\\0&I_{r-t} \end{smallmatrix}\right)W'(g)\chi_s(g) dg$$
where $dg$ is a Haar measure on $N_t\backslash G_t$.

Let's assume for motivation that we know $Z_s(W,W')$ lives in the field $k(s)$ of rational functions. In \cite{jps2}, in the case of $k=\CC$, this is proved using analytic arguments and finiteness results on Whittaker functions. With such a rationality result, the set of $Z_s(W,W')$ as $W$ and $W'$ vary form a fractional $k[s,s^{-1}]$-ideal of $k(s)$. Since $k[s,s^{-1}]$ is a PID, this set has a single generator which is the $L$-function and which captures all the poles. We can substitute anything in $k^{\times}$ lying outside this finite set of poles as a value for the indeterminate $s$. This specializes $\chi_s$ to a character $G\rightarrow k^{\times}$. In this case $Z_s$ defines a bilinear pairing
$$\cW(\tau)\times \cW(\tau'\otimes \chi_s)\rightarrow k.$$

In fact $Z_s(W,W')$ is rational. A proof of this for $k$ being any Noetherian $W(\overline{\mathbb{F}_{\ell}})$-algebra, and $\tau$ and $\tau'$ being co-Whittaker $k[G]$-modules, will appear in the forthcoming work \cite{moss2} of the author. The theory of co-Whittaker modules is developed in \cite{h_whitt}, and also in \cite{moss1}. In particular, absolutely irreducible generic representations over a field $k$ of characteristic $\ell\neq p$ are co-Whittaker.

To get a better handle on what $Z_s$ is actually doing, we investigate the functor defined by taking a Whittaker function $W(x)$ on $G_r$ to its restriction $W\left(\begin{smallmatrix}g&0\\0&I_{r-t} \end{smallmatrix}\right)$ to $G_t$.

\begin{lemma}[\cite{cogshap} Prop 1.1]
\label{explicitphi}
For $1\leq l \leq r-1$,
$$(\Phi^-)^l\tau \cong \left \{W\left(\begin{smallmatrix}   p & 0 \\ 0 & I_{l}  \end{smallmatrix}\right) : W \in \tau,\ p\in P_{r-l}\right\}.$$
\end{lemma}

Notice that the integral $Z_s(W,W')$ is determined in its first variable by the values $W(
\begin{smallmatrix}
g & 0 \\
0 & I_{r-t}
\end{smallmatrix}
)$ for $g\in G_t$.  But observe that as $W$ ranges over $\tau$, these values entirely determine the $P_{t+1}$-representation $$(\Phi^-)^{r-t-1}\tau \cong \left \{W\left(\begin{smallmatrix}   p & 0 \\ 0 & I_{r-t-1}  \end{smallmatrix}\right) : W \in \tau,\ p\in P_{t+1}\right\}$$
because for $p=(\begin{smallmatrix}g&0\\0&1\end{smallmatrix})(\begin{smallmatrix}I_t& x\\0&1\end{smallmatrix})$ in $P_{t+1} = G_tU_{t+1}$ where $U_{t+1}=\left(\begin{smallmatrix}&&& x_1\\& I_t &&\vdots \\ &&& x_t \\ 0 & \cdots &0 & 1\end{smallmatrix}\right)$, we have
$$W(\begin{smallmatrix}p&0\\0&I_{r-t-1}\end{smallmatrix}) = W\left((\begin{smallmatrix}g&0&0\\0&1&0\\0&0&I_{r-t-1}\end{smallmatrix})(\begin{smallmatrix}1& x&0\\0&1&0\\0&0&I_{r-t-1}\end{smallmatrix})\right)=\psi(\begin{smallmatrix}I_t&gx\\0&1\end{smallmatrix})W(\begin{smallmatrix}g&0\\0&I_{r-t}\end{smallmatrix}).$$ Considering the $P_{t+1}$-action by right translation on $(\Phi^-)^{r-t-1}\tau$, we have the following:
\begin{obs}
\label{restrictiontoGmodel}
The restriction of $(\Phi^-)^{r-t-1}\tau$ to the subgroup $G_t$ is isomorphic via the restriction of functions map to the space
$\left\{W\left(\begin{smallmatrix}   g & 0 \\ 0 & I_{r-t}  \end{smallmatrix}\right) : W \in \tau,\ g\in G_t\right\}$ in $\Rep(G_t)$.
\end{obs}

In particular, this shows that $Z_s$ is an element of $$\Hom_{G_t}(\Res^{P_{t+1}}_{G_t}(\Phi^-)^{r-t-1}\tau,[\cW(\tau'\otimes \chi_s)]^{\widetilde{\ \ \ }}).$$

\section{This Hom space is one-dimensional}
\label{homspace}

\begin{prop}
\label{uniquenessofzeta}
For $\tau$ in $\Rep_k(G_r)$ and $\tau'$ in $\Rep_k(G_t)$ both irreducible generic, $\Hom_{G_t}(\Res^{P_{t+1}}_{G_t}(\Phi^-)^{r-t-1}\tau,[\cW(\tau'\otimes \chi_s)]^{\widetilde{\ \ \ }})$ has dimension $\leq 1$ except for possibly finitely many values of $s$.
\end{prop}
\begin{proof}
Since $\tau$ is irreducible $\Res^{G_t}_{P_t}\tau$ is finite length. Hence it has a $P_r$-composition series $0=\tau_0\subset \cdots \subset \tau_l = \tau$. If $(\Phi^-)^{r-t-1}\tau=0$ we are done. Otherwise there is a nontrivial filtration $(\Phi^-)^{r-t-1}\tau_0\subset \cdots \subset (\Phi^-)^{r-t-1}\tau_l$. Fix an element $B$ in $\Hom_{G_t}((\Phi^-)^{r-t-1}\tau,[\cW(\tau')\otimes \chi_s]^{\widetilde{\ \ \ }})$. Then if $B$ is nonzero there is some $1\leq \lambda\leq l$ such that $B|_{(\Phi^-)^{r-t-1}\tau_{\lambda}}\not\equiv 0$ but $B|_{(\Phi^-)^{r-t-1}\tau_{\lambda-1}}\equiv 0$. Thus $B$ defines a nonzero element of $$\Hom_{G_t}((\Phi^-)^{r-t-1}\pi,[\cW(\tau')\otimes \chi_s]^{\widetilde{\ \ \ }})$$ where $\pi$ is the irreducible $P_r$ representation $\tau_{\lambda}/\tau_{\lambda-1}$. By Lemma \ref{hardlemma} below, the only way $B$ is nonzero in this situation is if $\pi$ is irreducible generic. 

Since $\tau$ is generic, the filtration $0=\tau_0^{(r)}\subset\cdots \subset \tau_l^{(r)}=\tau^{(r)}$ has exactly one nontrivial composition factor. Since we have discovered that $\pi^{(r)}=(\tau_{\lambda}/\tau_{\lambda-1})^{(r)}\neq 0$, this composition factor must be $\pi$.

Since $\tau_{\lambda-i}/\tau_{\lambda-i-1}$ is non-generic for $i\geq 1$ we can apply Lemma \ref{hardlemma} again to show that every element of $\Hom_{G_t}((\Phi^-)^{r-t-1}\tau|_{G_t},[\cW(\tau')\otimes \chi_s]^{\widetilde{\ \ \ }})$ must vanish on $(\Phi^-)^{r-t-1}\tau_{\lambda-i}$ for $i\geq 1$ since no nonzero homomorphisms can occur on $$(\Phi^-)^{r-t-1}(\tau_{\lambda-i}/\tau_{\lambda-i-1}).$$ Moreover, if $B$ vanishes on $(\Phi^-)^{r-t-1}\tau_{\lambda}$, it must vanish on all of $(\Phi^-)^{r-t-1}\tau$ (no nonzero homomorphisms can occur on $(\Phi^-)^{r-t-1}(\tau_{\lambda+i+1}/\tau_{\lambda+i})$ for $i\geq 0$).

Thus if there were another nonzero element $B'$ of $\Hom_{G_t}((\Phi^-)^{r-t-1}\tau,[\cW(\tau')\otimes \chi_s]^{\widetilde{\ \ \ }})$, there would be some number $c$ such that $B-cB'$ is zero on the subspace $(\Phi^-)^{r-t-1}\tau_{\lambda}$, hence zero on all of $(\Phi^-)^{r-t-1}\tau$. This proves that if the space is nonzero it is one-dimensional, except possibly (since we used Lemma \ref{hardlemma}) for finitely many values of $s$.
\end{proof}

\begin{lemma}
\label{hardlemma}
Let $\pi$ in $\Rep_k(P_r)$ be irreducible and $\tau'$ in $\Rep_k(G_t)$ be irreducible. Except for finitely many values of $s$, the space $\Hom_{G_t}((\Phi^-)^{r-t-1}\pi,[\tau'\otimes \chi_s]^{\widetilde{\ \ \ }})$ equals zero unless $\pi$ and $\tau'$ are generic in which case it has dimension 1.
\end{lemma}

\begin{proof}[Proof Attempt]
We use induction on $r$ starting with $r=2$, $t=1$. The base case is treated clearly in \cite[2.11 Lemma]{jps2}.

By \cite[3.5 Cor]{b-zI} since $\pi$ is irreducible we can write $\pi = (\Phi^+)^j\Psi^+\rho$ where $\rho$ is an irreducible representation of $G_{r-j-1}$.

Suppose $r-t-1 > j$. 

Then using \cite[3.2 Prop]{b-zI} we get
\begin{align*}
&\Hom_{G_t}((\Phi^-)^{r-t-1}\pi,\tau'\otimes \chi_s]^{\widetilde{\ \ \ }}) = \\
&\Hom_{G_t}((\Phi^-)^{r-t-1-j}\Psi^+(\rho),[\tau'\otimes \chi_s]^{\widetilde{\ \ \ }})=\\
&\Hom_{G_t}(0,[\tau'\otimes \chi_s]^{\widetilde{\ \ \ }})= 0,
\end{align*}
which shows that $B$ cannot have been nonzero.

Suppose $r-t-1=j$.

Then \cite[3.2 Prop]{b-zI} shows that
$$\Hom_{G_t}((\Phi^-)^{r-t-1}\pi|_{G_t},[\tau'\otimes \chi_s]^{\widetilde{\ \ \ }})  = \Hom_{G_t}(\Psi^+\rho|_{G_t},[\tau'\otimes \chi_s]^{\widetilde{\ \ \ }}).$$
But $\Psi^+\rho$ is the representation of $P_{t+1}$ by inflating $\rho$ to $U_{t+1}$ via the trivial character. Thus $\Psi^+\rho|_{G_t}=\rho$, and we have that $B$ defines an element of
$$\Hom_{G_t}(\rho,[\tau'\otimes \chi_s]^{\widetilde{\ \ \ }}).$$
Since $\rho$ and $\tau'$ are irreducible, this space is nonzero precisely when $\rho\cong  [\tau'\otimes\chi_s]^{\widetilde{\ \ \ }}$. Since the action of the group of unramified characters $F^{\times}\rightarrow k^{\times}$ on $\Irrep(G_{r-j-1})$ has finite stabilizers, we see that $\rho\cong \chi_s^{-1}\otimes\widetilde{\tau'}$ for only finitely many values of $s$.

Suppose $r-t-1< j$.

Using Lemma \ref{phiplusandrestriction} (proven below), we can write down the following chain of isomorphisms:
\begin{align*}
&\Hom_{G_t}(\Res^{P_{t+1}}_{G_t}(\Phi^-)^{r-t-1}\pi,[\tau'\chi_s]^{\widetilde{\ \ \ }}) &\text{expanding $\pi$}\\
&=\Hom_{G_t}(\Res^{P_{t+1}}_{G_t}(\Phi^+)^{j-(r-t-1)}\Psi^+\rho,[\tau'\chi_s]^{\widetilde{\ \ \ }}) & \text{\cite[3.2 Prop]{b-zI}} \\
&=\Hom_{G_t}(\tau', [\Res^{P_{t+1}}_{G_t}(\Phi^+)^{j-(r-t-1)}\Psi^+\rho \chi_s^{-1}]^{\widetilde{\ \ \ }})& \text{swapping the order} \\
&=\Hom_{G_t}(\tau',\Res^{P_{t+1}}_{G_t}(\hat{\Phi}^+)^{j-(r-t-1)}[\Psi^+\rho\chi_s^{-1}\Delta^{-1}]^{\widetilde{\ \ \ }}) &\text{\cite[2.25(c)]{b-z}}\\
&=\Hom_{G_t}(\tau',\Ind_{P_t}^{G_t}(\hat{\Phi}^+)^{j-(r-t)}[\Psi^+\rho\chi_s^{-1}\Delta^{-1}]^{\widetilde{\ \ \ }}) &\text{by Lemma \ref{phiplusandrestriction} below}\\
&=\Hom_{P_t}(\Res^{G_t}_{P_t}\tau',(\hat{\Phi}^+)^{j-(r-t)}[\Psi^+\rho\chi_s^{-1}\Delta^{-1}]^{\widetilde{\ \ \ }})&\text{by Frobenius reciprocity}\\
&=\Hom_{P_{r-j}}((\Phi^-)^{j-(r-t)}\tau', [\Psi^+\rho\chi_s^{-1}\Delta^{-1}]^{\widetilde{\ \ \ }}) & \text{\cite[3.2 Prop]{b-zI}}\\
\end{align*}
where $\Delta$ is the modulus character for an appropriate subgroup, depending only on $j-(r-t-1)$.

Consider the sub-case $j<r-1$.

Let $B'$ be the image of $B$ in $\Hom_{P_{r-j}}((\Phi^-)^{j-(r-t)}\tau', [\Psi^+\rho\chi_s^{-1}\Delta^{-1}]^{\widetilde{\ \ \ }})$. 

Since $\tau'$ is irreducible, $\tau'|_{P_t}$ is finite length with a decomposition series $0=\tau'_0\subset\cdots\subset\tau'_{l'}=\tau'$. If $(\Phi^-)^{j-(r-t)}\tau' = 0$ then we're done; otherwise there is a nontrivial filtration $(\Phi^-)^{j-(r-t)}\tau'_0\subset\cdots\subset(\Phi^-)^{j-(r-t)}\tau'_{l'}$. If $B$ were nonzero then $B'$ would also be nonzero hence $B'$ defines a homomorphism on 
$$\frac{(\Phi^-)^{j-(r-t)}\tau'_{\lambda}}{(\Phi^-)^{j-(r-t)}\tau'_{\lambda-1}} = (\Phi^-)^{j-(r-t)}(\tau'_{\lambda}/\tau'_{\lambda-1})$$ with $\lambda$ being the maximal index such that $B'|_{\tau'_{\lambda}}\not\equiv 0$. Denote by $\pi'$ the irreducible representation $\tau'_{\lambda}/\tau'_{\lambda-1}$, then $B'$ defines an element of 
$$\Hom_{P_{r-j}}((\Phi^-)^{j-(r-t)}\pi', [\Psi^+\rho\chi_s \Delta^{-1}]^{\widetilde{\ \ \ }}),$$ which embeds in
$$\Hom_{G_{r-j-1}}((\Phi^-)^{j-(r-t)}\pi'|_{G_{r-j-1}}, [\rho\chi_s\Delta^{-1}]^{\widetilde{\ \ \ }})$$ via forgetting the $U_{r-j}$-equivariance.
Since $\pi'$ is irreducible and $\rho$ is irreducible, we find by the induction hypothesis that (except for finitely many $s$ values) this space is zero except when $\pi'$ and $\rho$ are generic, in which case it is one dimensional. Thus it suffices to prove that $B'$ is zero when $\pi'\cong (\Phi^+)^{t-1}(1)$ and $\rho^{(n)}$ is one dimensional. Then, remembering the $U_{r-j}$ equivariance, we actually have
\begin{align*}
&\Hom_{P_{r-j}}((\Phi^-)^{j-(r-t)}\pi', [\Psi^+\rho\chi_s\Delta^{-1}]^{\widetilde{\ \ \ }})=\\
&\Hom_{P_{r-j}}((\Phi^+)^{r-j-1}(1), [\Psi^+\rho\chi_s\Delta^{-1}]^{\widetilde{\ \ \ }}) = 0
\end{align*}

Finally suppose $j=r-1$ (so in particular this is another sub-case of $r-j-1<t$).

Here we have $\rho = 1$. As before the space we are investigating becomes 
\begin{align*}
&\Hom_{P_{r-j}}((\Phi^-)^{j-(r-t)}\tau', [\Psi^+\rho\chi_s^{-1}\Delta^{-1}]^{\widetilde{\ \ \ }})=\\
&\Hom_k((\Phi^-)^{t-1}\tau', \chi_s^{-1}\Delta^{-1})
\end{align*}
Since $\tau'$ is irreducible, $\dim_k(\tau')^{(t)}\leq 1$. 

If $\Hom_{G_t}((\Phi^-)^{r-t-1}\pi,\widetilde{\tau' \chi_s})$ were nonzero, then (excepting finitely many values of $s$) the only possibility not yet excluded is the sub-case $j=r-1$. In this situation, $\tau'$ is generic and the space has dimension exactly 1. Converseley, if $\tau'$ and $\rho$ are generic then $j=r-1$ and $(\Phi^-)^{t-1}\tau'=(\tau')^{(t)} = k$, so $\Hom_k((\Phi^-)^{t-1}\tau', \chi_s^{-1}\Delta^{-1})$ is one dimensional.
\end{proof}

\section{The relation between $\hat{\Phi}^+$ and $\Res^{P_{t+1}}_{G_t}$}

We now prove the result regarding the relation between $\hat{\Phi}^+ = \Ind_{P_tU_{t+1}}^{P_{t+1}}$ and $\Res^{P_{t+1}}_{G_t}$, which was used for the $r-t-1<j$ case of the proof of Lemma \ref{hardlemma}.

\begin{lemma}
\label{phiplusandrestriction}
For $\sigma$ in $\Rep(P_t)$,  the map $\hat{\Phi}^+(\sigma) \rightarrow \Ind_{P_t}^{G_t}(\sigma)$ given by restricting functions to $G_t$ is an isomorphism of $G_t$-representations.
\end{lemma}
\begin{proof}
We will construct a left inverse given by extending functions to $U_{t+1}\subset P_{t+1}$. Given any element $f$ of  $\Ind_{P_t}^{G_t}(\sigma)$ we show there exists a unique element $f_{\psi}$ of $\Ind_{P_tU_{t+1}}^{P_{t+1}}(\sigma)$ satisfying $f_{\psi}|_{G_t}\equiv f$.

If we want to extend a function on $G_t$ to $P_{t+1}=G_tU_{t+1}$, and we want it to live in $\hat{\Phi}^+(\sigma)$, then we are left with no choice. Let $p=\left(\begin{smallmatrix}g&0\\0&1\end{smallmatrix}\right)\left(\begin{smallmatrix}& & & x_1\\&I_t&&\vdots\\ &&&x_t\\0&\cdots&0&1\end{smallmatrix}\right)$ be an element of $P_{t+1}$. Given $f\in \Ind_{P_t}^{G_t}(\sigma)$, we define its extension $f_{\psi}$ as follows:

\begin{align*}
f_{\psi}(p) &= f_{\psi}\left((\begin{smallmatrix}g&0\\0&1\end{smallmatrix})(\begin{smallmatrix}I_t& x\\0&1\end{smallmatrix})\right)\\
&= f_{\psi}\left((\begin{smallmatrix}g&0\\0&1\end{smallmatrix})(\begin{smallmatrix}I_t& x\\0&1\end{smallmatrix})(\begin{smallmatrix}g^{-1}&0\\0&1\end{smallmatrix})(\begin{smallmatrix}g&0\\0&1\end{smallmatrix})\right)\\
&=f_{\psi}\left((\begin{smallmatrix}I_t& gx\\0&1\end{smallmatrix})(\begin{smallmatrix}g&0\\0&1\end{smallmatrix})\right)\\
&=\psi\left(\begin{smallmatrix}I_t& gx\\0&1\end{smallmatrix}\right)f(g)\\
&=\psi(g_{t,1}x_1+\cdots + g_{t,t}x_t)f(g)
\end{align*}

When dealing with characters $\psi$ on $U_{t+1}$, the conjugation action of $G_t$ on $U_{t+1}$ gives an action of $G_t$ on $\widehat{U_{t+1}}$, which we will denote with a superscript $\psi^g$. In this notation, we can write the extension more succinctly:
\begin{equation}
\label{definingeqn}
f_{\psi}\left((\begin{smallmatrix}g&0\\0&1\end{smallmatrix})(\begin{smallmatrix}I_t& x\\0&1\end{smallmatrix})\right) := \psi^g(\begin{smallmatrix}I_t& x\\0&1\end{smallmatrix})f(g).
\end{equation}

We need to show $f_{\psi}$ defines an element of $\Ind_{P_tU_{t+1}}^{P_{t+1}}(\sigma)$. Recall that $P_t$ is the subgroup of $G_t$ that normalizes characters defined on the bottom entry of $U_{t+1}$, in particular $\psi^h \equiv \psi$ on $U_{t+1}$ if $h\in P_t$. Thus, using equation \ref{definingeqn} we have for $h\in P_t$,
\begin{align*}
f_{\psi}(hp) &=\psi^{hg}(\begin{smallmatrix}I_t & x\\0&1\end{smallmatrix})f(\begin{smallmatrix}hg & 0\\0&1\end{smallmatrix})\\
&= \psi^g(\begin{smallmatrix}I_t & x\\0&1\end{smallmatrix})\sigma(h)f(g)\\
&=\sigma(h)f_{\psi}(p),
\end{align*}
and for $v = (\begin{smallmatrix}1&y\\0&1\end{smallmatrix})\in U_{t+1}$ we have 
\begin{align*}
f_{\psi}(vp)&=f_{\psi}\left((\begin{smallmatrix}I_t&y\\0&1 \end{smallmatrix})(\begin{smallmatrix}g&0\\0&1 \end{smallmatrix})(\begin{smallmatrix}I_t&x\\0&1 \end{smallmatrix})\right)\\
&=f_{\psi}\left((\begin{smallmatrix}g&0\\0&1 \end{smallmatrix})(\begin{smallmatrix}I_t&x+g^{-1}y\\0&1 \end{smallmatrix})\right)\\
&=\psi^g(\begin{smallmatrix}I_t&x+g^{-1}y\\0&1 \end{smallmatrix})f(g)\\
&=\psi(v)f_{\psi}(p).
\end{align*}

Since $f_{\psi}(g)=f(g)$ for $g\in G_t$, we have defined a right inverse of the restriction of functions map.

We need to show it is $G_t$-equivariant with respect to the right translation actions, which we will denote by $\rho$ for both $\hat{\Phi}^+(\sigma)$ and $\Ind_{P_t}^{G_t}$. By the following calculation, $(\rho(g')f)_{\psi}\equiv \rho(\begin{smallmatrix}g'&0\\0&1\end{smallmatrix})(f_{\psi})$, proving that our map is $G_t$-equivariant:
\begin{align*}
(\rho(g')f)_{\psi}(p)&=(\rho(g')f)_{\psi}\left((\begin{smallmatrix}g&0\\0&1\end{smallmatrix})(\begin{smallmatrix}I_t& x\\0&1\end{smallmatrix})\right)\\
&=\psi^g(\begin{smallmatrix}I_t& x\\0&1\end{smallmatrix})(\rho(g')f)(g)\\
&=\psi^g(\begin{smallmatrix}I_t& x\\0&1\end{smallmatrix})f(gg')\\
&=\psi^{gg'}(\begin{smallmatrix}I_t & (g')^{-1}x\\0&1\end{smallmatrix})f(gg')\\
&= f_{\psi}\left((\begin{smallmatrix}gg' & 0\\0&1\end{smallmatrix})(\begin{smallmatrix}I_t & (g')^{-1}x\\0&1\end{smallmatrix})\right)\\
&=f_{\psi}\left((\begin{smallmatrix}g&0\\0&1 \end{smallmatrix})(\begin{smallmatrix}I_t&x\\0&1 \end{smallmatrix})(\begin{smallmatrix}g'&0\\0&1\end{smallmatrix})\right)\\
&=f_{\psi}(pg')\\
&=\rho(\begin{smallmatrix}g'&0\\0&1\end{smallmatrix})(f_{\psi})(p)\text{, for all } p=(\begin{smallmatrix}g&0\\0&1\end{smallmatrix})(\begin{smallmatrix}I_t& x\\0&1\end{smallmatrix})\in P_{t+1}
\end{align*}
\end{proof}
By definition, applying $\Phi^-$ to a $G_t$ representation entails first restricting to $P_t$ and then applying the functor $\Phi^-:\Rep(P_t)\rightarrow \Rep(P_{t-1})$. In fact, it is equivalent to compactly induce to $P_{t+1}$ and apply $\Phi^-$ twice:
\begin{cor}
\label{phiandinduction}
For $\tau'\in \Rep(G_t)$, $(\Phi^-)^{m+1}(\cInd_{G_t}^{P_{t+1}}\tau')\cong (\Phi^-)^m(\tau')$ for $0\leq m\leq t$.
\end{cor}
\begin{proof}
We take $(\Phi^-)^0$ to equal $\Res^{G_t}_{P_t}$. It suffices to prove the corollary for $m=0$. By Frobenius reciprocity and by Lemma \ref{phiplusandrestriction} below, we have the following isomorphisms for any $\sigma$ in $\Rep(P_t)$: 
\begin{align*}
\Hom_{P_t}(\Phi^-\cInd_{G_t}^{P_{t+1}}\tau',\sigma)&=\Hom_{G_t}(\tau',\Res^{P_{t+1}}_{G_t}\hat{\Phi}^+(\sigma))\\
&=\Hom_{G_t}(\tau',\Ind_{P_t}^{G_t}(\sigma))\\
&=\Hom_{P_t}(\Res^{G_t}_{P_t}\tau',\sigma),
\end{align*}
hence $\Phi^-(\cInd_{G_t}^{P_{t+1}}\tau')\cong \Res^{G_t}_{P_t}\tau'$ by the Yoneda lemma.
\end{proof}

In particular, taking $m=t$ we have $(\cInd_{G_t}^{P_{t+1}}\tau')^{(t+1)}\cong (\tau')^{(t)}$.

\begin{rmk}
The proof of Lemma \ref{phiplusandrestriction} shows that restriction of functions gives a $G_t$ isomorphism from the Kirillov space $\Ind_{N_{t+1}}^{P_{t+1}}\psi$ to the Whittaker space $\Ind_{N_t}^{G_t}\psi$ (typically the dual result appears for restriction of functions from $G_t$ to $P_t$).
\end{rmk}

\bibliography{mybibliography}{}
\bibliographystyle{alpha}
\end{document}